\documentclass[11pt]{article}
 \usepackage{amsmath} 
 \usepackage{amsthm} 
 \usepackage{amssymb}
 \usepackage[T1]{fontenc} 
 \usepackage{comment} 
 \usepackage[english]{babel} 
 \theoremstyle{plain}
 \newtheorem{thm}{Theorem}
 \newtheorem{defin}{Definition}
 \newtheorem*{thm*}{Theorem}

 \theoremstyle{remark}

 \newtheorem*{rmks*}{Remarks}

 \numberwithin{equation}{section}
 \usepackage{graphicx}

  \newcommand{\lla}{\left\langle} 
   \newcommand{\rra}{\right\rangle} 

 \usepackage{hyperref}

\begin{document}
 	\author{Susanna Risa  \footnote{Dipartimento di Matematica, Universit\`a di Roma ``Tor Vergata'', Via della Ricerca Scientifica 1, 00133, Roma, Italy. E-mail: risa@mat.uniroma2.it}
 		, Carlo Sinestrari \footnote{Dipartimento di Ingegneria Civile e Ingegneria Informatica, Universit\`a di Roma ``Tor Vergata'', Via Politecnico 1, 00133, Roma, Italy. E-mail: sinestra@mat.uniroma2.it}
 	}
 	\title{Strong spherical rigidity of ancient solutions of expansive curvature flows\footnotetext{Published on the Bulletin of the London Mathematical Society, Volume 52, Issue 1 (pp. 94-99). doi:10.1112/blms.12308}}
 	\date{}
 	\maketitle
 	\begin{abstract}
 		We consider geometric flows of hypersurfaces expanding by a function of the extrinsic curvature and we show that the homothethic sphere is the unique solution of the flow which converges to a point at the initial time. The result does not require assumptions on the speed other than positivity and monotonicity and it is proved using a reflection argument. Our theorem shows that expanding flows exhibit stronger spherical rigidity, if compared with the classification results of ancient solutions in the contractive case. 	\end{abstract}\medskip
 
 	\section{Introduction}
 	In this short article, we consider expansive curvature flows of hypersurfaces in Euclidean space. The flow is given by a time-dependent family of smooth embeddings $\varphi: M \times (T_0,T_1) \to \mathbb{R}^{n+1}$ of an $n$-dimensional closed manifold $M$, which evolves according to
 	\begin{equation} \label{ICF}
 	\frac{\partial \varphi}{\partial t}=\frac{\nu}{F}.
 	\end{equation}
 	Here $\nu$ is the outer normal vector and $F$ is a positive, symmetric function of the principal curvatures, strictly increasing with respect to each argument. 
The most important example is the inverse mean curvature flow, see e.g. \cite{HIRP, HI}, where $F$ is given by the mean curvature, which has found relevant applications such as the proof of the Riemannian Penrose inequality in general relativity \cite{HIRP}, or an estimate for the smooth Yamabe invariants of a class of $3$-manifolds \cite{BrayNeves}. In general, such flows are parabolic and they possess a solution for small positive times starting from any given smooth initial hypersurface for which $\frac 1F$ is well-defined. Rather than the Cauchy problem, we focus here on certain special solutions of the flow, an example of which is given by the so-called ancient solutions, defined in time intervals of the form $(-\infty,T_1)$. Such solutions typically arise as limit flows in rescaling procedures, see for example \cite{H2}, thus they play an important role in the asymptotic analysis of the evolution.

 A standard sphere expands by homotheties under \eqref{ICF}, since the equation translates into an ODE for the radius. Backwards in time, the radius becomes smaller and converges to zero as $t \downarrow T_0$, for some $T_0$ which is either finite or infinite depending on the growth rate of $F$. In this paper, we show that the expanding sphere is the only compact embdedded solution of a flow of the form \eqref{ICF} which "comes out of a point", meaning that it converges to a point in the Hausdorff sense as $t \downarrow T_0$. No a priori assumption of asymptotical roundness, or convergence of rescalings as $t \downarrow T_0$ is required. The proof relies on a powerful result by Chow and Gulliver \cite{CG1, CG2}, which adapts Alexandrov's reflection technique to the setting of geometric flows. As in \cite{CG1, CG2}, our result holds under very general hypotheses on $F$: besides positivity and monotonicity, no assumptions such as homogeneity, or convexity/concavity are needed. As a particular case, our theorem implies that the spheres are the only homothetically expanding compact solutions of these flows, a result which was first proved in \cite{DLW} for the inverse mean curvature flow and in \cite{CCF} for other $1$-homogeneous flows.
 
Although the proof is a relatively simple application of the reflection technique, our result is somehow unexpected, in the light of the recent studies on ancient solutions of contractive flows, i.e. evolution equations where the right hand side in \eqref{ICF} has the form $-F\nu$, with $F$ as above.
Under such flows, spheres contract homothetically and always give rise to ancient solutions. On the other hand, various nontrivial examples of compact ancient solutions have been constructed during the years, e.g. \cite{An1, BLT,XJWang}. In particular, we recall the White Ovaloids \cite{Ang, HH,W03}: they are convex solutions which, as $t\to -\infty$, become more and more eccentric, but sweep the whole space. In analogy, one could conjecture the existence of solutions shaped like ellipsoids also in the expanding case, with axes going to zero at different rates as $t \to -\infty$. However, this is excluded by the result of this paper: a nonspherical convex ancient solution, if it exists, should come out of a set which is larger than a point, e.g. from a segment. To our knowledge, no compact ancient solution of \eqref{ICF} besides the spherical one has currently been described.

In the contractive case, various authors have found characterizations of the shrinking spheres as the unique ancient solutions which satisfy suitable geometric conditions, such as uniform convexity, uniform pinching of inner and outer radii, or a bound on the diameter growth, see \cite{HH, HS, LL, RS}. Our theorem can be regarded as a counterpart in the expanding case of these rigidity results for those speeds which admit ancient flows, since we obtain that the sphere is the only compact embedded ancient solution which comes out of a point. Nevertheless, a comparison shows that the rigidity of the spherical solution in the expanding case is much stronger.
In fact, the arguments of \cite{HH,HS,LL,RS} depend crucially on the properties of the specific classes of speeds considered in each case, such as degree of homogeneity, convexity or concavity, behaviour at the boundary of the cone of definition or specific properties of the equations involved. In addition, all geometric hypotheses used to characterize the spherical solution are replaced here by the weaker property of coming out of a point. We also observe that our rigidity result in the expanding case only depends on this geometric assumption, regardless if the initial time is finite or not, and thus it also holds for those speeds $F$ whose growth is not compatible with the existence of compact ancient solutions.

It is interesting to notice, on the other hand, that there are similarities with the results for contractive flows in the spherical ambient space $\mathbb{S}^{n+1}$ obtained by Bryan-Louie \cite{BrLo} and Bryan-Ivaki-Scheuer \cite{BIS} for curves and higher dimensional hypersurfaces respectively. In this setting, there exists an ancient solution, called spherical cap, given by a family of shrinking geodesic spheres, which converge to an equator as $t \to -\infty$. These authors have established the uniqueness of the spherical cap among all ancient solutions for general speeds, under the only assumption of backwards convergence to an equator. We observe that the arguments in \cite{BIS, BrLo} also rely on a reflection technique, although different than the one used here, as well as on the symmetries of the limiting equator. We also mention the papers by Bourni, Langford and Tinaglia \cite{BLT, BLT2}, where the result by Chow and Gulliver has been applied to prove axial symmetry of certain nonspherical ancient solutions of the Mean Curvature Flow.
 	
 	\section{Rigidity of solutions coming out of a point}
 
 In what follows, we consider the flow \eqref{ICF}, where the speed $F$ is a smooth symmetric function of the principal curvatures $\lambda_1,\dots,\lambda_n$ and is defined on an open convex symmetric cone $\Gamma \subset \mathbb{R}^n$ which contains the positive diagonal, i.e. the $n$-tuples of the form $(\lambda, \dots,\lambda)$ for some $\lambda>0$. We assume
\begin{equation}\label{ipotesi}
F>0, \quad \quad \frac{\partial F}{\partial \lambda_i}>0  \quad\quad \forall \, i=1,\dots,n, \quad \mbox{ everywhere on }\Gamma.
\end{equation}

We consider a smooth solution of the flow \eqref{ICF}, defined on a time interval  $t \in (T_0,T_1)$, with $-\infty \leq T_0 < T_1 \leq +\infty$, and we denote the evolving hypersurface by $\varphi(M,t)=M_t$. By saying that $M_t$ is a solution of \eqref{ICF}, we assume in particular that $M_t$ is admissible for $F$, meaning that the principal curvatures of $M_t$ belong to $\Gamma$ for all $t$. In addition, we denote by $\Omega_t$ the bounded open domain whose boundary is $M_t$. Our positivity assumption on $F$ implies that the flow is expansive, that is, $M_{t_1} \subset \Omega_{t_2}$ whenever $t_1<t_2$. 

 	\begin{defin}\label{.oO}
 		Let $M_t$ be a solution of \eqref{ICF}. We say that $M_t$ comes out of a point $y_{\infty} \in \mathbb{R}^{n+1}$ if for every $r>0$ there exists a time $\overline{t} \in (T_0,T_1)$ such that $M_{\overline{t}}\subset B_r(y_{\infty})$, where $B_r(y_{\infty})$ is the ball of radius $r$ centered at $y_{\infty}$.
 	\end{defin}

 \begin{rmks*}~
 	\begin{enumerate}
 	\item Any spherical solution to \eqref{ICF} has a radius $r(t)$ which evolves according to $\dot r(t)=1/(\psi(r(t)))$,  where we have defined $\psi(r)=F\left(\frac 1r,\dots,\frac 1r\right)$. Therefore, spherical solutions are ancient if and only if
	\begin{equation}\label{growth}
	\int_0^1 \psi(r)\,dr=+\infty.
	\end{equation} 
 	In particular, if $F$ is an $\alpha$-homogeneous function of the principal curvatures, a sphere is ancient if and only if $\alpha \geq 1$. Flows with homogeneous speed have been widely investigated in \cite{GerICF, GerICFp, UrbICF}, where convergence forwards in time to a sphere has been established for starshaped or convex hypersurfaces.
 	\item If condition \eqref{growth} is not satisfied, then the flow \eqref{ICF} does not admit any compact ancient solution. Suppose in fact that there exists a compact solution $M_t$ defined on an interval $(-\infty, T_1)$, and fix any $T<T_1$. Let $\rho_-(T)$ be the radius of 
 the largest ball enclosed by $M_T$, and consider the spherical solution with radius $\rho_-(T)$ at time $T$. By assumption, this solution 
comes out of a point at a finite time $T_S< T$. Since $M_t$ is defined and smooth for all negative times, it encloses some sphere with positive radius at time $T_S$. By comparison, the evolution of this other sphere remains inside $M_t$ for all following times. On the other hand, by construction, its radius at time $T$ is strictly larger than $\rho_-(T)$, which gives a contradiction.
 	\item A similar argument proves that if the sphere is ancient, then any solution coming out of a point is also ancient.
 	\item  The condition in Definition \ref{.oO} is trivially satisfied by any expanding soliton. In addition, for an ancient solution, it is
	implied by several natural (but stronger) geometric assumptions on the evolving submanifold, such as any of the following:
	\begin{enumerate}
	\item a uniform pinching $\rho_-(t)\geq\epsilon_0 \rho_+(t)$ between inner and outer radius,
	\item a uniform curvature pinching $\min\limits_i \lambda_i \geq \epsilon_0 \max\limits_i \lambda_i >0$ on $M_t$, 
	\item uniform starshapedness $\langle \varphi,\nu \rangle \geq \epsilon_0 |\varphi|$ on $M_t$,
	\end{enumerate}
	where $\epsilon_0>0$ is a constant independent of $t$. On the other hand, our condition is compatible with quite different behaviours: it is satisfied, for example, by an ellipsoid whose axes go to zero at different rates as $t \to T_0$, so that (a),(b),(c) would fail.
 \end{enumerate}
 \end{rmks*}
 
The theorem we prove in this paper shows that the property of coming out of a point is particularly rigid.
 	\begin{thm}\label{maintheo}
	Let $F:\Gamma \to \mathbb{R}$ be a smooth symmetric function satisfying \eqref{ipotesi}. Then, 
 		any smooth, closed, embedded solution of \eqref{ICF} coming out of a point in $\mathbb{R}^{n+1}$ is a family of expanding spheres.
 	\end{thm}

 	Our proof is based on the reflection technique introduced by Alexandrov in his work on CMC surfaces \cite{Aleks}, readapted for curvature flows in a series of papers by Chow and Gulliver \cite{ChowRef, CG1, CG2}. We follow the notation in \cite{ChowRef}.
 	\begin{defin}
 	For $V\in\mathbb{S}^{n}$, $c\in \mathbb{R}$, we denote by $\pi_V^c$ the plane \linebreak $\pi_V^c=\{y \in \mathbb{R}^{n+1}\,|\, \lla y,V\rra=c\}$.
 	Let $H^+(\pi_V^c)$ (resp. $H^-(\pi_V^c)$) be the halfspace $\{y\in \mathbb{R}^{n+1}\,|\,\lla y,V \rra > c\}$ (resp. $\{y\in \mathbb{R}^{n+1}\,|\,\lla y,V \rra < c\}$). Let $M$ be a hypersurface in $\mathbb{R}^{n+1}$ which bounds the domain $\Omega$ and let $M^{\pi}$ be the reflection of $M$ about $\pi_V^c$, $M^{\pi}=\{y-2(\lla y,V \rra-c)V\,|\,y\in M\}$.
 	 We say that $M$ can be strictly reflected at $\pi_V^c$ if $M^{\pi}\cap H_-(\pi_V^c) \subset \Omega \cap H_-(\pi_V^c)$ and $V$ is not tangent to $M$ along $M \cap \pi_V^c$. 
	 \end{defin}
 	Then, the following result holds:
 	\begin{thm}[Chow-Gulliver \cite{ChowRef,CG2}]\label{CGtheo}
 		Let $\varphi:M\times [0,T) \rightarrow \mathbb{R}^{n+1}$ be an embedded $C^2$ solution of \eqref{ICF}. For $V\in\mathbb{S}^{n}$, $c\in \mathbb{R}$, let $\pi_V^c$ be the plane $\pi_V^c=\{y \in \mathbb{R}^{n+1}\,|\, \lla y,V\rra=c\}$.
 		If $M_0$ can be reflected strictly at $\pi_V^c$, then $M_t$ can be reflected strictly at $\pi_V^c$ for all times $t \in [0,T)$.
 		\end{thm}

After these preliminaries, we can give the proof of Theorem 1.		
		 \begin{proof}[Proof of Theorem \ref{maintheo}]
 Without loss of generality, we will assume that $y_{\infty} = 0$. We will show that the assumption of coming out of a point implies that our solution can be strictly reflected at every plane which is close to the origin, but does not contain it, if time is sufficiently close to $T_0$. We then conclude by a limit procedure that $M_t$ is symmetric with respect to every plane through the origin.
 	This will imply that the solution is a sphere for all times.

We begin by taking an arbitrary $t_* \in (T_0,T_1)$. Since the solution is smooth and embedded, there exists $R_* >0$ such that
\begin{equation}\label{tstar}
B_{R_*}(0) \subset \Omega_{t_*}.
\end{equation}
Now let us fix any direction $V \in \mathbb{S}^n$ and any $c \in \mathbb{R}$ with $0< c < R_*$. By \eqref{tstar}, the hyperplane $\pi_V^c$ 
intersects the solution at time $t_*$. On the other hand, the assumption that $M_t$ comes out of a point implies that $M_t$ does not intersect $\pi_V^c$ if $t$ is sufficiently close to $T_0$. Since $M_t$ is compact, there exists a first time $\tau=\tau(V,c) \in (T_0,t^*)$ at which $M_t$ touches the hyperplane, i.e.
$$M_\tau \cap \pi_V^c\neq \emptyset, \qquad M_t \subset H_-(\pi_V^c) \mbox{ for all }t<\tau.$$
Using the property that $M_t$ comes out of a point we also find that, for any fixed $V$,
\begin{equation}\label{tau}
\tau(V,c) \to T_0 \mbox{ as }c \to 0.
\end{equation}

We claim that, for $t$ larger than $\tau$ but sufficiently close to $\tau$,  the solution $M_t$ can be strictly reflected at $\pi_V^c$. Intuitively, the property follows from the tangency of $M_t$ to $\pi_V^c$ at time $t=\tau$ and from the expansive character of the flow. We now proceed to give a precise motivation of the claim.

We denote the contact set at time $\tau$ as $M^c=M_\tau \cap \pi_V^c$. Then we have $\nu(p, \tau) = V$ for $p \in M$ such that $\varphi(p,\tau) \in M^c$.
    We fix $0<\epsilon<<1$ and define $M_\nu=\{p\in M\,|\, |\nu(p,\tau)-V|<\epsilon\}$ and the complement of its image $I_\nu=M_\tau \setminus \varphi(M_{\nu},\tau)$. The set $I_\nu$ is compact in $\mathbb{R}^{n+1}$ and does not intersect $\pi_V^c$, thus the distance $d=dist(I_\nu,\pi_V^c)$ is positive.
    
    We fix $\delta > 0$ such that, for all $t \in [\tau,\tau+\delta]$ and for all $p \in M$:
    \begin{align}
    |\varphi(p,t)-\varphi(p,\tau)|&<\frac d4,\label{condist}\\
    |\nu(p,t)-\nu(p,\tau)|&< \epsilon. \label{norcond}
    \end{align}

    These conditions are clearly satisfied if $\delta$ is small enough, since $M_t$ is smooth and compact, and so both $\varphi$ and $\nu$ are uniformly continuous on compact intervals of time.
        
    Condition \eqref{norcond} ensures, for $p \in M_{\nu}$, that 
    \begin{equation}\label{disnorm}
    |\nu(p,t)-V|\leq |\nu(p,t)-\nu(p,\tau)|+|\nu(p,\tau)-V|<2\epsilon.
    \end{equation}
 Since $\nu$ is the outer normal, property \eqref{disnorm} implies that, for all $p \in M_\nu$ there exists $\sigma_0>0$ such that
 \begin{equation}\label{segmeq}
  \varphi(p,t) +\sigma V \notin \Omega_t, \quad \varphi(p,t) -\sigma V \in \Omega_t, \quad \forall\, 0<\sigma<\sigma_0
 \end{equation}
 
 	Now, we consider $M_t^+=M_{t}\cap H_+(\pi_V^c)$, for $t \in (\tau, \tau+\delta]$: it is the part of the evolving hypersurface which has moved across the plane $\pi_V^c$  and it is nonempty, as the normal to any point in $M^c$ is $V$ and $F>0$ holds. By \eqref{condist}, all points of $M_t^+$ have distance less than $d/4$ from $\pi_V^c$. On the other hand, if $p \notin M_\nu$, its image $\varphi(p,t)$ is at distance at least $\frac 34 d$ from $\pi_V^c$, again by \eqref{condist} and by the definition of $d$. It follows that $M_t^+$ is contained in the image of $M_{\nu}$ under the immersion $\varphi(\cdot,t)$. Then \eqref{disnorm} shows that the normal $\nu$ is close to $V$ at all points of $M_t^+$, and we see in particular that, at the points where $M_{t}$ intersects $\pi_V^c$, $V$ does not belong to the tangent space to $M_t$.

 We then want to show that all points in the reflection of $M_t^+$ about $\pi_V^c$ are contained in $\Omega_t$ if $t \in (\tau, \tau+\delta]$.  As already observed, both $M_t^+$ and its reflection lie at distance less than $d/4$ from the plane $\pi_V^c$. We then fix any $y \in M_t^+$ and consider the segment $y - \sigma V$, with $\sigma \in [0,2d_y]$, where $d_y=dist(y,\pi_V^c)$. By \eqref{segmeq}, the points of the segment lie in the enclosed region $\Omega_t$ for small $\sigma>0$. If there is a first $\bar \sigma \in (0,2d_y]$ such that $y - \bar \sigma V \in M_t$, the same arguments as above imply that $y - \bar \sigma V$ is the image of a point in $M_\nu$, and therefore \eqref{segmeq} gives a contradiction. Therefore the segment is entirely contained in $\Omega_t$.
 
Since $y \in  M_t^+$, this shows that $M_{t}$ can be reflected at $\pi_V^c$ for all $ t \in (\tau, \tau+\delta]$; Theorem \ref{CGtheo} ensures that the same is true for all following times. The same argument holds for every $0< c < R_*$; by letting $c \to 0$ and using \eqref{tau}, 
we conclude that $M_{t}$ can be reflected, possibly nonstrictly, at the plane $\pi_V^0$ passing through the origin, for all $t> T_0$. Since this is true for any direction $V$, this implies that $M_t$ is symmetric about any plane through the origin and it is thus a sphere.
 	
 	 \end{proof}
  \bibliographystyle{abbrv}
  \bibliography{biblio2}
 	\end{document}